\documentclass{amsart}

\newcommand{\pr}{\mathbb{P}}

\newcommand{\Ol}{\mathcal{O}}

\newcommand{\ra}{\rightarrow}

\usepackage{amsmath}
\usepackage{amssymb, amsfonts,amsthm,amscd}
\usepackage{latexsym,amsthm,amscd,flafter,graphicx,verbatim,stmaryrd,mathrsfs,color}
\usepackage[all]{xy}
\usepackage{mathrsfs}
\usepackage{nicefrac}

\begin{document}

\newtheorem{theorem}{Theorem}[section]
\newtheorem{lemma}[theorem]{Lemma}
\newtheorem{definition}[theorem]{Definition}
\newtheorem{proposition}[theorem]{Proposition}
\newtheorem{corollary}[theorem]{Corollary}
\newtheorem{remark}[theorem]{Remark}

\title{Representability of cohomological functors over extension fields}
\author{\bf{Alice Rizzardo} }
\subjclass[2010]{18E30,14F05} 
\keywords{representability, base extension, Fourier-Mukai}
\maketitle

\begin{abstract}
We generalize a result of Orlov and Van den Bergh on the representability of a cohomological functor  $H:D^{b}_{\mathrm{Coh}}(X)\to \underline{\mathrm{mod}}_{L}$ to the case where $L$ is a field extension of the base field $k$ of the variety $X$, with $\text{trdeg}_k L\leq 1$ or $L$ purely transcendental of degree 2.

This result can be applied to investigate the behavior of an exact functor $F:D^{b}_{\mathrm{Coh}}(X)\to D^{b}_{\mathrm{Coh}}(Y)$ with $X$ and $Y$ smooth projective varieties and $\text{dim } Y\leq 1$ or $Y$ a rational surface. We show that for any such $F$ there exists a ``generic kernel'' $A$ in $D^{b}_{\mathrm{Coh}}(X\times Y)$, such that $F$ is isomorphic to the Fourier-Mukai transform with kernel $A$ after composing both with the pullback to the generic point of $Y$.
\end{abstract}

%%%%%%%%%%%%%%%%%%%%%%%%%%%%%%%%%%%%%%%%%%%%%
%%%%%%%%%%%%%%%%%%%%%%%%%%%%%%%%%%%%%%%%%%%%%
\section{Introduction}%%%%%%%%%%%%%%%%%%%%%%%%
%%%%%%%%%%%%%%%%%%%%%%%%%%%%%%%%%%%%%%%%%%%%%
%%%%%%%%%%%%%%%%%%%%%%%%%%%%%%%%%%%%%%%%%%%%%

Let $X$ be a smooth projective variety over an algebraically closed field $k$.  In this paper we will generalize a result of Orlov and Van den Bergh  \cite[Lemma 2.14]{failure} on the representability of a functor $H:D^{b}_{\mathrm{Coh}}(X)\to \underline{\mathrm{mod}}_{k}$ to the case of an extension field $k \subset L$:
\begin{theorem}\label{ebounded}
Let $X$ be a smooth projective variety over a field $k$. Let $L$ be a finitely generated separable field extension of $k$ with $\text{trdeg}_{k}L\leq 1$, or a purely transcendental field extension of transcendence degree 2 over $k$. Consider a contravariant, cohomological, finite type functor
$$H:D^{b}_{\mathrm{Coh}}(X)\ra \underline{\mathrm{mod}}_{L}$$
Then $H$ is representable by an object $E\in D^{b}_{\mathrm{Coh}}(X_{L})$, i.e. there exists $E$ such that for every $C\in D^{b}_{\mathrm{Coh}}(X)$ we have
$$H(C)=\mathrm{Mor}_{D^{b}_{\mathrm{Coh}}(X_{L})}(j^{*}C,E)$$
where $j^{*}:X_{L}\to X$ is the base change morphism.
\end{theorem}

An interesting example of a functor as in Theorem \ref{ebounded} can be obtained from an exact functor $F:D^{b}_{Coh}(X)\to D^{b}_{\mathrm{Coh}}(Y)$ between the bounded derived categories of two smooth projective varieties $X$ and $Y$, where $\text{dim } Y\leq 1$ or $Y$ is a rational surface. To produce a functor as in the above theorem, we compose $F$ with the pullback to the generic point of $Y$, take cohomology, then dualize to get a contravariant functor:

$$\xymatrix{
D^{b}_{\mathrm{Coh}}(X) \ar@/_2
pc/ [rrrr]_{H}  \ar[r]^-{F} & D^{b}_{\mathrm{Coh}}(Y) \ar[r]^-{i^{*}} & D^{b}_{\mathrm{Coh}}(\eta) \ar[r]^{H^{0}} & \underline{\mathrm{mod}}_{K(Y)} \ar[r]^{D} &\underline{\mathrm{mod}}_{K(Y)}
}$$

Theorem \ref{ebounded} will thus allow us to tackle the question of whether a functor between the bounded derived categories of two smooth projective varieties is representable by a Fourier-Mukai transform. When $\text{dim }Y\leq 1$ or $Y$ is a rational surface we can answer positively to the question above after restricting to the generic point of $Y$: 

\begin{theorem}\label{FM}
Let $X$, $Y$ be smooth projective varieties over a field $k$, where $\dim Y\leq 1$ or $Y$ is a rational surface. Consider a covariant exact functor
$$F: D^{b}_{\mathrm{Coh}}(X) \rightarrow D^{b}_{Coh}(Y)$$
let $i: \eta\ra Y$ the inclusion of the generic point of $Y$. Then there exists an object $A\in D^{b}_{\mathrm{Coh}}(X\times Y)$ such that 
$$ i^{*}\circ F =  i^{*}\circ \Phi_{A},$$
where $\Phi_{A}(\cdot):=Rp_{2*}(A\stackrel{L}{\otimes} Lp_{1}^{*}(\cdot))$ is the Fourier-Mukai trasform with kernel $A$ and $p_{1}: X\times Y\to X$, $p_{2}:X\times Y\to Y$ are the projection morphisms.
\end{theorem}

For some results generalizing the results in Section \ref{base change section}, see \cite{scalar}.

\subsection*{Acknowledgements}This paper is derived from part of my PhD thesis. I would like to sincerely thank my thesis advisor Aise Johan de Jong for suggesting this problem as well as providing invaluable guidance over the years. I am grateful to Michel Van den Bergh for many helpful comments and suggestions. Finally, I would like to thank MSRI for its hospitality during the Noncommutative Algebraic Geometry and Representation Theory semester, when the final draft of this paper was completed.

%%%%%%%%%%%%%%%%%%%%%%%%%%%%%%%%%%%%%%%%%%%%%
%%%%%%%%%%%%%%%%%%%%%%%%%%%%%%%%%%%%%%%%%%%%%
\section{The Base Change Category}%%%%%%%%%%%
%%%%%%%%%%%%%%%%%%%%%%%%%%%%%%%%%%%%%%%%%%%%%
%%%%%%%%%%%%%%%%%%%%%%%%%%%%%%%%%%%%%%%%%%%%%

In what follows, an abelian category $\mathcal{A}$ does not automatically have any limits or colimits apart from the finite ones. 

Given a field $K$, we will denote with $\underline{\mathrm{mod}}_{K}$ the category of finite dimensional $K$-vector spaces, whereas $\underline{\mathrm{Mod}}_{K}$ will denote the category of possibly infinite-dimensional $K$-vector spaces. $D(\mathcal{A})$ will denote the derived category of an abelian category $\mathcal{A}$.

Given an $R$-linear abelian category $\mathcal{A}$ and an inclusion of rings $R\hookrightarrow S$, we can define the base change category $\mathcal{A}_{S}$ as in \cite[\textsection 4]{deformation}:
\begin{definition}
The category $\mathcal{A}_{S}$ is given by pairs $(C,\rho_{C})$ where $C\in \mathrm{Ob}(\mathcal{A})$ and $\rho_{C}:S\to \mathrm{Hom}_{\mathcal{A}}(C,C)$ is an $R$-algebra map such that the composition $R\to S \to \mathrm{Hom}_{\mathcal{A}}(C,C)$ gives back the $R$-algebra structure on $\mathcal{A}$. The morphisms in $\mathcal{A}_{S}$ are the morphisms in $\mathcal{A}$ compatible with the $S$-structure.
\end{definition}

\begin{definition}\label{tensoring}
For  each element $C\in \mathcal{A}$, the functor
$$C \otimes_{R}-: \underline{\mathrm{mod}}(R)\to \mathcal{A}$$
is the unique finite colimit preserving functor with $C\otimes R=C$.

This gives for each finitely presented $R$-algebra $S$ a functor
$$-\otimes S:\mathcal{A} \ra \mathcal{A}_{S}$$
to the base change category $\mathcal{A}_{S}$.
\end{definition}

\begin{proposition}\cite[Proposition 4.3]{deformation}\label{tensor adjointness}
The functor $-\otimes S$ is left adjoint to the forgetful functor
\begin{align*}
\mathrm{Forget}:\mathcal{A}_{S} &\ra \mathcal{A}\\
(C,\rho_{C}) &\mapsto C
\end{align*}
\end{proposition}
Whenever the context is clear, given an object $B\in \mathcal{A}_{S}$, we will still denote by $B$ the corresponding object of $\mathcal{A}$ obtained via the forgetful functor.

For the purposes of this discussion we will need a more general setting for base change - specifically, we need to be able to talk about base change for a bigger category of rings and not just the ones that are finitely presented over the base. Let us extend Definition \ref{tensoring} as follows:

\begin{definition}\label{abelian tensoring}
Let $\mathcal{A}$ be an $R$-linear abelian category satisfying AB5. Using the fact that any $R$-module is the filtered colimit of finitely presented $R$-modules, we can extend definition \ref{tensoring} to the general case of 
$$-\otimes S:\mathcal{A} \ra \mathcal{A}_{S}$$
for any $R$-algebra $S$.
\end{definition}

The notion of base change category can be extended to the case of the derived category $D(\mathcal{A})$ of an abelian $R$-linear category $\mathcal{A}$ in the obvious way:
\begin{definition}
Given an inclusion of rings  $R\hookrightarrow S$, the category $D(\mathcal{A})_{S}$ is given by pairs $(C,\rho_{C})$ where $C\in \mathrm{Ob}(D(\mathcal{A}))$ and $\rho_{C}:S\to \mathrm{Hom}_{D(\mathcal{A})}(C,C)$ is an $R$-algebra map such that the composition $R\to S \to \mathrm{Hom}_{D(\mathcal{A})}(C,C)$ gives back the $R$-algebra structure on $D(\mathcal{A})$. The morphisms in $D(\mathcal{A})_{S}$ are the morphisms in $D(\mathcal{A})$ compatible with the $S$-structure.
\end{definition}

Again, we have a notion of tensor product:
\begin{definition}\label{more tensoring}
Let $R$ be a ring, let $\mathcal{A}$ be an $R$-linear abelian category satisfying AB5, and let $C^{\bullet}$ be a complex of objects in $\mathcal{A}$:
$$C^{\bullet}=\ldots\ra C^{i-1}\xrightarrow{d^{i-1}} C^{i} \xrightarrow{d^{i}} C^{i+1}\ra\ldots$$
Let $S$ be a ring, with a map $R\hookrightarrow S$. Then we can define $C^{\bullet}\otimes S$, as an object of $D(\mathscr{A}_{S})$, as
$$C^{\bullet}\otimes S=\ldots\ra C^{i-1}\otimes S\xrightarrow{d^{i-1}\otimes 1} C^{i}\otimes S \xrightarrow{d^{i}\otimes 1} C^{i+1}\otimes S\ra\ldots$$
\end{definition}
The complex $C^{\bullet}\otimes S$ can also be considered as an object of $D(\mathscr{A})_{S}$.

\begin{remark}
Suppose that $\mathcal{A}$ is a $k$-linear abelian category satisfying AB5 and $k\subset K$ is an extension of fields. In the situation of definition \ref{abelian tensoring} and \ref{more tensoring}, similarly to the case of \ref{tensor adjointness}, it is easy to show that again tensoring with $K$ is left adjoint to the forgetful functor
\begin{itemize}
\item as a functor $\mathcal{A}\to \mathcal{A}_{K}$;
\item as a functor $D(\mathcal{A})\to D(\mathcal{A}_{K})$;
\item as a functor $D(\mathcal{A})\to D(\mathcal{A})_{K}$.
\end{itemize}
\end{remark}

\begin{remark}
Let $R$ be a ring, let $\mathcal{A}$ be an $R$-linear abelian category satisfying AB5, and let $C^{\bullet}$ be a complex of objects in $\mathcal{A}$,
$$C^{\bullet}=\ldots\ra C^{i-1}\xrightarrow{d^{i-1}} C^{i} \xrightarrow{d^{i}} C^{i+1}\ra\ldots$$
Let $S\subset R$ a multiplicative system. In this case $C^{\bullet}\otimes_{R}S^{-1}R$, as an object of $D(\mathscr{A})$, is the same as
$$\ldots\ra \mathop{\mathrm{colim}}_{f\in S}f^{-1}C^{i-1}\xrightarrow{d^{i-1}} \mathop{\mathrm{colim}}_{f\in S}f^{-1}C^{i} \xrightarrow{d^{i}} \mathop{\mathrm{colim}}_{f\in S}f^{-1} C^{i+1}\ra\ldots$$

where $\mathop{\mathrm{colim}}_{f\in S}f^{-1}C^{i}$ is obtained by taking for every $f \in S$ a copy of $C^{i}$
and as morphisms only the maps
$$
f^{-1}\,C^i
\longrightarrow
(fg)^{-1}\,C^i
$$
given by multiplication by $g: C^{i} \to C^{i}$.
\end{remark}

\begin{lemma}\label{tensor}
In the situation of the remark above, if for every element $f\in S$ the multiplication by $f$ is a quasi-isomorphism of $C^{\bullet}$, then the map $$C^{\bullet}\ra C^{\bullet}\otimes_{R}S^{-1}R$$
is a quasi-isomorphism in $D(\mathcal{A})$.
\end{lemma}
\begin{proof}
Since taking cohomology commutes with directed colimits we have 
$$H^{i}(C^\bullet \otimes_{R} S^{-1}R)= \mathop{\text{colim}}_{f\in S}\ f^{-1}\,H^{i}(C^\bullet)$$
but since multiplication by any $g\in S$ is a quasi-isomorphism we get
$$\xymatrix{
f^{-1}\,H^{i}(C^{\bullet})\ar[r]^-{\cong}_-{g} & (fg)^{-1}\,H^{i}(C^{\bullet})
}$$
hence the cohomology of $C^\bullet \otimes_{R} S^{-1}R$ consists of only one copy of $H^{i}(C^{\bullet})$, and the map $C^{\bullet}\ra C^\bullet \otimes_{R} S^{-1}R$ is a quasi-isomorphism.
\end{proof}

%%%%%%%%%%%%%%%%%%%%%%%%%%%%%%%%%%%%%%%%%%%%%
%%%%%%%%%%%%%%%%%%%%%%%%%%%%%%%%%%%%%%%%%%%%%
\section{A result on base change for derived categories}\label{base change section}%%%
%%%%%%%%%%%%%%%%%%%%%%%%%%%%%%%%%%%%%%%%%%%%%
%%%%%%%%%%%%%%%%%%%%%%%%%%%%%%%%%%%%%%%%%%%%%

The purpose of this section is to analyze the functor $D(\mathcal{A}_{K}) \ra D(\mathcal{A})_{K}$ that sends an object in $D(\mathcal{A}_{K})$ to the same object considered as an object of $D(\mathcal{A})$, together with its $K$-action. Specifically, we will prove the following:

\begin{theorem}\label{ess-sur}
Let $\mathcal{A}$ be a $k$-linear abelian category satisfying AB5, where $k$ is a field. Let $K=k(T)$ or $K=k(T,T')$. Then the functor
\begin{align*}
D(\mathcal{A}_{K}) &\ra D(\mathcal{A})_{K}\\
C^{\bullet} &\mapsto (C^{\bullet},\rho_{C})
\end{align*}
is essentially surjective, where $\rho_{C}:K\to \mathrm{Aut}(C^{\bullet})$ is the obvious map.

Moreover, if $L$ is a finite separable extension of $K=k(T)$ with $L=K(\alpha)=K[T]/P(T)$ then we can lift an object $(C^{\bullet}, \rho_{C})\in D(\mathscr{A})_{L}$ to an object $N^{\bullet}$ of $D(\mathscr{A}_{K})$ endowed with a map $\psi_{\alpha}\in \mathrm{End}(N^{\bullet})$ such that $P(\psi_{\alpha})$ is zero on all cohomology groups, and the action of $\psi_{\alpha}$ on $N^{\bullet}$ corresponds to the action of $\alpha$ on $C^{\bullet}$.
\end{theorem}

A stronger results for the case of a finite extension $K/k$ was obtained in \cite{Pawel}. In this case, there is actually an equivalence $D(\mathcal{A}_{K}) \to D(\mathcal{A})_{K}$.

The proof of this theorem will be carried out in several steps. First we will notice that, in the purely transcendental case $K=k(T,T')$, this comes down to lifting the actions of the two variables $T$ and $T'$ on a complex $C^{\bullet}\in D(\mathcal{A})_{K}$, given by $\rho_{C}(T)$ and $\rho_{C}(T')$,  to actions coming from morphisms in $\mathcal{A}$ that commute with each other. 

Then in Lemma \ref{simpler} we will tackle the case of one variable and obtain a complex $M^{\bullet}\in D(\mathcal{A}_{k[T]})$ with a quasi-isomorphism to $C^{\bullet}$ as objects of $D(\mathcal{A})$, and such that the $T$-actions on $M^{\bullet}$ and $C^{\bullet}$ cohincide. At this point, since $\rho_{C}(T)$ is an automorphism of $C^{\bullet}$, tensoring with $k(T)$ will give us a complex in $D(\mathcal{A}_{k(T)})$ which is still quasi-isomorphic to $C^{\bullet}$.

A similar process can be repeated twice, as we will show in Lemma \ref{simpler2}.

%%%%%%%%%%%%%%%%%%%%%

\begin{lemma}\label{equivalence}
Let $e^{n}D(\mathcal{A})$ be the category whose
\begin{enumerate}
\item Objects are pairs $(C, \varphi_{1},\dots,\varphi_{n})$ where $E \in \mathrm{Ob}(D(\mathcal{A}))$, $\varphi_{i} \in \mathrm{End}_{D(\mathcal{A})}(C)$ for all $i$, and $\varphi_{i}$ commutes with $\varphi_{j}$ for all $i,j$;
\item Morphisms $a : (C, \varphi_{1},\dots,\varphi_{n}) \to (C', \varphi'_{1},\dots,\varphi'_{n})$ are elements
$a \in \mathrm{Hom}_{D(\mathcal{A})}(C, C')$ such that $a \circ \varphi_{i} = \varphi'_{i} \circ a$.
\end{enumerate}
Consider the full subcategory $e^{n}D'(\mathcal{A}) \subset e^{n}D(\mathcal{A})$
whose objects consist
of those pairs $(C, \varphi_{1},\dots,\varphi_{n})$ such that for every nonzero $f \in k[T_{1},\ldots,T_{n}]$
the map $f(\varphi_{1},\dots,\varphi_{n}) : C \to C$ is an isomorphism in $D(\mathcal{A})$. 

The category $D(\mathcal{A})_{k(T_{1},\ldots,T_{n})}$ is equivalent to the category $e^{n}D'(\mathcal{A})$. The equivalence is given by the functor
$$
D(\mathcal{A})_{k(T_{1},\ldots,T_{n})} \longrightarrow e^{n}D'(\mathcal{A}),
\quad
(C, \rho_C) \longmapsto (C, \rho_C(T_{1}),\ldots, \rho_C(T_{n})).
$$
\end{lemma}

\begin{proof}
The equivalence is given by the inverse functor
\begin{align*}
e^{n}D'(\mathcal{A}) &\longrightarrow D(\mathcal{A})_{k(T_{1},\ldots,T_{n})} \\
(C,\varphi_{1},\ldots,\varphi_{n}) & \mapsto 
\left(C, \begin{array}{rcl} \rho_{C}:k(T_{1},\ldots,T_{n}) & \ra & \mathrm{Aut}©\\ T_{i}& \mapsto & \varphi_{i} \end{array}\right)
\end{align*}

\end{proof}

%%%%%%%%%%%%%%%%%%%%%%%%%%%%%%%%%%%%

\begin{lemma}\label{simpler}
Let $\mathcal{A}$ be a $k$-linear abelian category satisfying AB5, where $k$ is a field. Let $C^\bullet$ be a complex in $D(\mathcal{A})$. Let $\varphi \in \mathrm{Hom}_{D(\mathcal{A})}(C^\bullet, C^\bullet)$. Then there exists a complex
$M^{\bullet}\in D(\mathcal{A}_{k[T]})$ and a quasi-isomorphism $C^\bullet \to M^\bullet$ as objects of $D(\mathcal{A})$ such that the action of multiplication by $T$ on $M^\bullet$ corresponds to the action of multiplication by $\varphi$
on $C^\bullet$.
\end{lemma}

Note that when $\mathcal{A}$ is a Grothendieck category, the same result can be achieved by considering the morphism corresponding to $\varphi$ on a $K$-injective replacement of $C^{\bullet}$ (which exists by \cite[Prop. 3.2]{brown}) and defining the action of $T$ accordingly. However, in what follows we will need this specific form for the complex $M^{\bullet}$.

\begin{proof}
The map $\varphi: C^{\bullet}\ra C^{\bullet}$ in $D(\mathcal{A})$ corresponds to a diagram of complexes in $D(\mathcal{A})$
$$\xymatrix{
&Q^{\bullet} \ar[dl]_{u} \ar[dr]^{\varphi'} \\
C^{\bullet} \ar@{-->}[rr]^-{\varphi} & &C^{\bullet} 
}$$
where $u$ is a quasi-isomorphism.

Let $C^{\bullet}[T]=C^{\bullet}\otimes_{k}k[T]$ as a complex in $D(\mathcal{A}_{k[T]})$. Consider the morphism $\varphi \otimes 1 - 1 \otimes T: C^{\bullet}[T] \ra C^{\bullet}[T]$ in $D(\mathcal{A}_{k[T]})$. This can be represented by actual maps of complexes
$$\xymatrix{
&Q^{\bullet}[T] \ar[dl]_{u\otimes 1} \ar[dr]^{\varphi' \otimes 1 - u \otimes T} \\
C^{\bullet}[T] \ar@{-->}[rr]^{\varphi \otimes 1 - 1 \otimes T} & &C^{\bullet} [T]
}$$

The map $\varphi' \otimes 1 - u \otimes T$ is injective on all cohomology objects: to prove this we need to show that $\varphi' \otimes 1 - u \otimes T:H^{r}(Q^{\bullet}[T])\ra H^{r}(C^{\bullet}[T])$ is injective for every $r$.

Let $\alpha \in H^{r}(Q^{\bullet}[T])$, $\alpha\neq 0$, then
$$\alpha=\sum_{i=0}^{n}\alpha_{i} T^{i}$$
where all of the $\alpha_{i}$ are different from zero in $H^{r}(Q^{\bullet})$. If 
$$0=(\varphi' \otimes 1 - u \otimes T)\alpha= \sum_{i=0}^{n}\varphi'(\alpha_{i}) T^{i}-\sum_{i=0}^{n}u(\alpha_{i})T^{i+1}$$
then the only term of degree $n+1$ in $T$, $u(\alpha_{n})T^{n+1}$, must be zero in $H^{r}(C^{\bullet})$, hence $u(\alpha_{n})=0$, hence $\alpha_{n}=0$ since $u$ is a quasi-isomorphism. This contradicts our assumption that $\alpha_{i}\neq 0\;\forall i$, and so this proves injectivity.

Now set
$$
M^\bullet
=
\text{Cone}
(Q^\bullet[T] \xrightarrow{\varphi' \otimes 1 - u \otimes T} C^\bullet[T])
$$

Then we have a distinguished triangle
\begin{equation}\label{triangle1}
Q^\bullet[T]
\xrightarrow{\varphi' \otimes 1 - u \otimes T}
C^\bullet[T]
\longrightarrow M^\bullet
\longrightarrow (Q^\bullet[T])[1]
\end{equation}
and by injectivity of the map $\varphi' \otimes 1 - u \otimes T$ on the cohomology objects we get a short exact sequence in cohomology
$$0\ra H^{r}(Q^{\bullet}[T])\xrightarrow{\varphi' \otimes 1 - u \otimes T}H^{r}(C^{\bullet}[T])\longrightarrow H^{r}(M^{\bullet})\ra 0$$
hence we get 
$$H^{r}(M^{\bullet})=\text{Coker}( H^{r}(Q^{\bullet}[T])\xrightarrow{\varphi' \otimes 1 - u \otimes T}H^{r}(C^{\bullet}[T]))$$
for any $r$.

Now consider the composition 
$$
C^\bullet \longrightarrow C^\bullet[T] \xrightarrow{c} M^\bullet
$$

This map is a quasi-isomorphism; to prove this we just need to show that under the map above, $H^{r}(C^{\bullet}) \cong \text{Coker}( H^{r}(Q^{\bullet}[T])\xrightarrow{\varphi' \otimes 1 - u \otimes T}H^{r}(C^{\bullet}[T])) $ for every $r$.
 
Proceed as follows: first of all, considered as a sub-object of $H^{r}(C^{\bullet}[T])$ via the obvious map $C^{\bullet}\to C^{\bullet}[T]$, $H^{r}(C^{\bullet})$ is not in the image of $\varphi' \otimes 1 - u \otimes T$, since, for any element $\alpha=\sum_{i=1}^{n}\alpha_{i} T^{i}$ of $H^{r}(Q^{\bullet}[T])$, its image $\sum_{i=1}^{n}\varphi(\alpha_{i}) T^{i}-\sum_{i=0}^{n}u(\alpha_{i})T^{i+1}$ is either zero or has a nonzero term of positive degree. To prove that any term of positive degree $\beta=\sum_{i=1}^{n}\beta_{i} T^{i}$ is in the image up to an element of degree zero, notice that it can be written as an element of lower degree plus an element of the image as follows:
\begin{align*}
\sum_{i=0}^{n}\beta_{i} T^{i} &=\sum_{i=0}^{n}\beta_{i} T^{i}-(\varphi' \otimes 1 - u \otimes T)(u^{-1}(\beta_{n})T^{n-1})+(\varphi' \otimes 1 - u \otimes T)(u^{-1}(\beta_{n})T^{n-1})= \\
&= \sum_{i=0}^{n}\beta_{i} T^{i}-\varphi'(u^{-1}(\beta_{n}))T^{n-1}+\beta_{n}T^{n} +(\varphi' \otimes 1 - u \otimes T)(u^{-1}(\beta_{n})T^{n-1})\\
&=\sum_{i=0}^{n-1}\beta_{i} T^{i}-\varphi'(u^{-1}(\beta_{n}))T^{n-1}+(\varphi' \otimes 1 - u \otimes T)(u^{-1}(\beta_{n})T^{n-1})
\end{align*}

Hence we found a complex $M^{\bullet}\in D(\mathcal{A}_{k[T]})$ which is quasi-isomorphic to $C^{\bullet}$ as an object of $D(\mathcal{A})$; moreover the action of multiplication by $\varphi$
on $C^\bullet$ corresponds to the action by multiplication by $T$ on $M^\bullet$, because the following diagram is commutative in $D(\mathcal{A})$:
$$
\xymatrix{
C^\bullet \ar[r] \ar[d]^{\varphi} & C^\bullet[T] \ar[r]^{c}\ar[d]^{\varphi\otimes 1} & M^\bullet \ar[d]^{T}\\
C^\bullet \ar[r] & C^\bullet[T] \ar[r]^{c} & M^\bullet
}
$$
this follows from the fact that $c\circ (1\otimes T) - (\varphi \otimes 1)\circ c = (1\otimes T)\circ c - (\varphi \otimes 1)\circ c = (1\otimes T - \varphi \otimes 1)\circ c =0$ since those are two consecutive maps in a triangle. 
\end{proof}

%%%%%%%%%%%%%%%%%%%%%%%

\begin{lemma}\label{simpler2}
Let $\mathcal{A}$ be a $k$-linear abelian category satisfying AB5, where $k$ is a field. Let $C^\bullet$ be a complex in $D(\mathcal{A})$. 

Let $\varphi \in \mathrm{Hom}_{D(\mathcal{A})}(C^\bullet, C^\bullet)$ such that $f(\varphi)$ is an isomorphism for all $f\in k[T]$ monic. Then there exists a complex
$N^{\bullet}\in D(\mathcal{A}_{k(T)})$ and a quasi-isomorphism $C^\bullet \to N^\bullet$ as objects of $D(\mathcal{A})$ such that the action of multiplication by $T$ on $N^\bullet$ corresponds to the action by multiplication by $\varphi$ on $C^\bullet$.

Likewise, let $\varphi, \psi \in \mathrm{Hom}_{D(\mathcal{A})}(C^\bullet, C^\bullet)$ such that $\varphi$ and $\psi$ commute with each other and such that $f(\varphi,\psi)$ is a quasi-isomorphisms for all $f\in k[T,T']$ nonzero. Then there exists a complex
$N^{\bullet}\in D(\mathcal{A}_{k(T,T')})$ and a quasi-isomorphism
$j : C^\bullet \to N^\bullet$ as objects of $D(\mathcal{A})$ such that the action of multiplication by $T$ (resp. $T'$) on $N^\bullet$ corresponds to the action by multiplication by $\varphi$ (resp. $\psi$) on $C^\bullet$.\end{lemma}

\begin{proof}
By Lemma \ref{simpler} we can find a complex $M^{\bullet}\in \mathcal{A}_{k[T]}$ and a quasi-isomorphism $j : C^\bullet \to M^\bullet$ as objects of $D(\mathcal{A})$ such that the action of multiplication by $T$ on $M^\bullet$ corresponds to the action by multiplication by $\varphi$ on $C^\bullet$. This implies that multiplication by $f(T)$ gives a quasi-isomorphism of $M^{\bullet}$ for all $f$ monic.

Now let $N^{\bullet}:=M^\bullet \otimes_{k[T]} k(T)$ as in Definition \ref{more tensoring} above. This is a complex in $D(\mathcal{A}_{k(T)})$ and it is quasi-isomorphic to $C^{\bullet}$ as objects of $D(\mathcal{A})$, by Lemma \ref{tensor}. The action of $\varphi$ on $C^{\bullet}$ corresponds to the action of $T$ on $N^{\bullet}$.

For the second case, again by Lemma \ref{simpler} we can find a complex $M^{\bullet}\in \mathcal{A}_{k[T]}$ and a quasi-isomorphism
$j : C^\bullet \to M^\bullet$ as objects of $D(\mathcal{A})$ such that the action of multiplication by $T$ on $M^\bullet$ corresponds to the action by multiplication by $\varphi$ on $C^\bullet$.

Moreover, we have an exact triangle
$$C^{\bullet}[T] \xrightarrow{\varphi\otimes 1-1\otimes T}C^{\bullet}[T]\longrightarrow M^{\bullet}$$
in $D(\mathcal{A}_{k[T]})$, see (\ref{triangle1}).

Then, since $\varphi$ and $\psi$ commute with each other, we get a diagram in $D(\mathcal{A}_{k[T]})$:
$$
\xymatrix{
C^\bullet[T] \ar[rr]^{\varphi \otimes 1 - 1 \otimes T}
\ar[d]_{\psi \otimes 1} & &
C^\bullet[T] \ar[d]^{\psi \otimes 1} \\
C^\bullet[T] \ar[rr]^{\varphi \otimes 1 - 1 \otimes T} & &
C^\bullet[T]
}
$$
This diagram is commutative: this follows from the fact that $\varphi\circ \psi=\psi \circ \varphi$ in $D(\mathcal{A})$, hence $\varphi \psi\otimes 1=\psi \varphi\otimes 1$ in  $D(\mathcal{A}_{k[T]})$. Therefore we can find a map $\tilde{\psi}$ on $M^{\bullet}$ so that the following diagram commutes:
$$
\xymatrix{
C^\bullet[T] \ar[rr]^{\varphi \otimes 1 - 1 \otimes T}
\ar[d]_{\psi\otimes 1} & & C^\bullet[T] \ar[d]^{\psi\otimes 1}  \ar[r] &
M^\bullet \ar[d]^{\tilde{\psi}} \ar[r] & (C^\bullet[T])[1] \ar[d]^{\psi\otimes 1} \\
C^\bullet[T] \ar[rr]^{\varphi \otimes 1 - 1 \otimes T} & &
C^\bullet[T] \ar[r] &
M^\bullet \ar[r] & (C^\bullet[T])[1]
}
$$
it follows that the action of $\tilde{\psi}$ on $M^{\bullet}$ is the same as the action of $\psi$ on $C^{\bullet}$, thanks to the commutativity of 
$$
\xymatrix{
C^\bullet \ar[r] \ar[d]^{\psi} & C^\bullet[T] \ar[r]\ar[d]^{\psi\otimes 1} & M^\bullet \ar[d]^{\tilde{\psi}}\\
C^\bullet \ar[r] & C^\bullet[T] \ar[r] & M^\bullet
}
$$
taking into account the fact that, as we mentioned already, the composition of the two horizontal maps is a quasi-isomorphism. As before we can then construct $P^{\bullet}=M^\bullet \otimes_{k[T]} k(T)$, which is quasi-isomorphic to $M^{\bullet}$ and hence we get a corresponding map $\tilde{\psi}: P^{\bullet}\to P^{\bullet}$.

So we are in the following situation: we have a complex $P^{\bullet}\in D(\mathcal{A}_{k(T)})$ and a map $\tilde{\psi}:P^{\bullet}\ra P^{\bullet}$ so that $f(\tilde{\psi})$ is a quasi-isomorphism for all $f\in k(T)[T']$ monic. By Lemma \ref{simpler} again, we get a complex $Q^{\bullet}\in D((\mathcal{A}_{k(T)})_{k(T)[T']})=D(\mathcal{A}_{k(T)[T']})$ which is quasi-isomorphic to $P^{\bullet}$.

Then define
$$N^{\bullet}:=Q^{\bullet}\otimes_{k(T)[T']}k(T,T')$$

By Lemma \ref{tensor}, since $f(T,\psi)$ is a quasi-isomorphisms for all nonzero $f\in k(T)[T']$, the complex $N^{\bullet}\in D(\mathcal{A}_{k(T,T')})$ is quasi isomorphic to $Q^{\bullet}$ as objects of $D(\mathcal{A}_{k(T)[T']})$ hence it is quasi-isomorphic to $C^{\bullet}$ as objects of $D(\mathcal{A})$. The action of $\varphi$ and $\psi$ correspond to the action of $T$ and $T'$ respectively.
\end{proof}

The last thing we need to do is to tackle the case of a general separable field extension of transcendence degree one, corresponding to the last statement of Theorem \ref{ess-sur}:

%%%%%%%%%%%%%%%%%%%%%%%%%%%

\begin{lemma}\label{simpler3}
Let $\mathcal{A}$ be a $k$-linear abelian category satisfying AB5, where $k$ is a field. Let $C^\bullet$ be a complex in $D(\mathcal{A})$. Let $\varphi, \psi \in \mathrm{Hom}_{D(\mathcal{A})}(C^\bullet, C^\bullet)$ such that $\varphi$ and $\psi$ commute with each other, and such that $f(\varphi)$ is a quasi-isomorphisms for all $f\in k[T]$ monic and there exists an irreducible $P\in k[T,T']$ with $P(\varphi, \psi)=0$. 

Then there exists a complex $N^{\bullet}\in D(\mathcal{A}_{k(T)})$ and a quasi-isomorphism $j : C^\bullet \to N^\bullet$ as objects of $D(\mathcal{A})$ such that the action of multiplication by $T$ on $N^\bullet$ corresponds to the action by multiplication by $\varphi$ on $C^\bullet$. Moreover there is a morphism $\tilde{\psi}\in \mathrm{End}(N^{\bullet})$ such that the action of $\psi$ on $C^{\bullet}$ corresponds to the action of $\tilde{\psi}$ on $N^{\bullet}$ and $P(T,\tilde{\psi})$ induces the zero map on all cohomology groups of $N^{\bullet}$.
\end{lemma}

\begin{proof}
By Lemma \ref{simpler} we can find a complex $M^{\bullet}\in \mathcal{A}_{k[T]}$ and a quasi-isomorphism
$j : C^\bullet \to M^\bullet$ as objects of $D(\mathcal{A})$ such that the action of multiplication by $T$ on $M^\bullet$ corresponds to the action by multiplication by $\varphi$
on $C^\bullet$.

Moreover, we have an exact triangle
$$C^{\bullet}[T] \xrightarrow{\varphi\otimes 1-1\otimes T}C^{\bullet}[T]\longrightarrow M^{\bullet}$$
in $D(\mathcal{A}_{k[T]})$, see (\ref{triangle1}).

Then, since $\varphi$ and $\psi$ commute with each other, we get a commutative diagram in $D(\mathcal{A}_{k[T]})$
$$
\xymatrix{
C^\bullet[T] \ar[rr]^{\varphi \otimes 1 - 1 \otimes T}
\ar[d]_{\psi \otimes 1} & &
C^\bullet[T] \ar[d]^{\psi \otimes 1} \\
C^\bullet[T] \ar[rr]^{\varphi \otimes 1 - 1 \otimes T} & &
C^\bullet[T]
}
$$

Therefore we can find a map $\tilde{\psi}$ on $M^{\bullet}$ so that the following diagram commutes:
$$
\xymatrix{
C^\bullet[T] \ar[rr]^{\varphi \otimes 1 - 1 \otimes T}
\ar[d]_{\psi\otimes 1} & & C^\bullet[T] \ar[d]^{\psi\otimes 1}  \ar[r] &
M^\bullet \ar[d]^{\tilde{\psi}} \ar[r] & (C^\bullet[T])[1] \ar[d]^{\psi\otimes 1} \\
C^\bullet[T] \ar[rr]^{\varphi \otimes 1 - 1 \otimes T} & &
C^\bullet[T] \ar[r] &
M^\bullet \ar[r] & (C^\bullet[T])[1]
}
$$

Since $P(T,\psi)=0$, we obtain that $P(\varphi\otimes 1, \psi\otimes 1)=0$ in $D(\mathscr{A}_{k[T]})$, hence $P(T,\psi\otimes 1)$ is zero on $C[T]$. 

As before we can construct $N^{\bullet}=M^\bullet \otimes_{k[T]} k(T)$, which is quasi-isomorphic to $M^{\bullet}$ and hence we get a corresponding map $\tilde{\psi}: N^{\bullet}\to N^{\bullet}$ and the action of $\psi$ on $C^{\bullet}$ corresponds to the action of $\tilde{\psi}$ on $N^{\bullet}$. 

Finally, since $P(T,\psi\otimes 1)$ is zero on $C[T]$, it follows that $P(T,\tilde{\psi})=0$ induces the zero map on all cohomology of $M^{\bullet}$ and hence of $N^{\bullet}$.
%I do not get zero on the level of complexes, because two maps being zero on a triangle do not guarantee that the third map will also be zero
\end{proof}

%%%%%%%%%%%%%%%%%%%
We are now ready to prove Theorem \ref{ess-sur}:

\begin{proof}[Proof of Theorem \ref{ess-sur}]
By Lemma \ref{equivalence}, we just need to show that the functors
\begin{align*}
D(\mathcal{A}_{K}) &\ra e^{1}D'(\mathcal{A})\\
C^{\bullet} &\mapsto (C^{\bullet},\cdot T)
\end{align*}
and
\begin{align*}
D(\mathcal{A}_{K}) &\ra e^{2}D'(\mathcal{A})\\
C^{\bullet} &\mapsto (C^{\bullet},\cdot T,\cdot T')
\end{align*}
are essentially surjective.

Let $(E,\varphi)\in e^{1}D'(\mathcal{A})$. Then by Lemma \ref{simpler2} there exists $N^{\bullet}\in \mathcal{A}_{k(T)}$ such that $N$ is quasi isomorphic to $E$ and the action of $\varphi$ on $E^{\bullet}$ corresponds to the action of $T$ on $N^{\bullet}$. This proves the case $i=1$.

Similarly, let $(E,\varphi,\varphi')\in e^{2}D'(\mathcal{A})$. Then by Lemma \ref{simpler2} there exists $N^{\bullet}\in \mathcal{A}_{k(T,T')}$ such that $N$ is quasi isomorphic to $E$ and the action of $\varphi$ and $\varphi'$ on $E^{\bullet}$ correspond to the action of $T$ and $T'$ respectively on $N^{\bullet}$. This proves the case $i=2$.

The last part follows from Lemma \ref{simpler3} by setting $\psi_{\alpha}:=\tilde{\psi}$.
\end{proof}

%%%%%%%%%%%%%%%%%%%%%%%%%
Let us now apply this theorem to the case $\mathcal{A}=\text{QCoh}(X)$, where $X$ is a quasi-compact, separated scheme over a field $k$. This is possible since $\text{QCoh}(X)$ satisfies AB5. Moreover, note that in this case we have an equivalence $D_{\text{Qcoh}}(X)\cong D(\text{Qcoh}(X))$. As a preliminary step, we will prove the following technical lemma:

\begin{lemma}\label{hartshorne}
Let $k\subset K$ be a field extension, $X$ a quasi-compact and separated scheme. Let $X_{K}\stackrel{j}{\ra}X$ the base change morphism. Then there is an equivalence of categories 
$$D_{\mathrm{QCoh}}(X_{K}) \stackrel{\Psi}{\longrightarrow} D(\mathrm{QCoh}(X)_{K})$$
under this equivalence, the functors 
$$Lj^*,\cdot\otimes K:D_{\mathrm{QCoh}}(X) \ra D(\mathrm{QCoh}(X_{K}))$$
and 
$$Rj_{*},\mathrm{Forget}: D_{\mathrm{QCoh}}(X_{K}) \ra D(\mathrm{QCoh}(X))$$
coincide.

In other words, 
\begin{align*}
Rj_{*} =\mathrm{Forget} \circ \Psi:D(\mathrm{QCoh}(X)_{K}) &\ra D_{\mathrm{QCoh}}(X)\\
\Psi \circ Lj^* =-\otimes K: D_{\mathrm{QCoh}}(X) &\ra (D_{\mathrm{QCoh}}(X))_{K}
\end{align*}

This is summarized in the following diagram:
$$\xymatrix{
D_{\mathrm{QCoh}}(X) \ar@/^/ [d]^{\otimes K} \ar@/^5pc /[dd]^{Lj^*} \\
(D_{\mathrm{QCoh}}(X))_{K} \ar@/^/[u]^{\mathrm{Forget}} \\
D_{\mathrm{QCoh}}(X_{K})=D(\mathrm{QCoh}(X_{K})) \ar[u]_{\Psi} \ar@/^5pc/ [uu]^{j_{*}}
}$$
\end{lemma}

\begin{proof}
There is an equivalence of categories induced by $j_{*}$ between quasi-coherent $\Ol_{X_{K}}$-modules and quasi-coherent $j_{*} \Ol_{X_{K}}$-modules on X. But $j_{*} \Ol_{X_{K}}=\Ol_{X}\otimes K$ and an $(\Ol_{X}\otimes K)$-module is the same thing as an $\Ol_{X}$-module with a $K$-structure which is compatible with its $k$-structure.

Hence we get an equivalence
\begin{align*}
\Psi: \mathrm{QCoh}(X_{K}) &\to \text{QCoh}(X)_{K} \\
C &\mapsto (j_{*}C,\rho_{C})
\end{align*}
where $\rho_{C}$ is the composition $K\to \Ol_{X}\otimes K\to \mathrm{End}(j_{*}C)$.

Under this equivalence, the two functors $j_{*}$ and ``Forget'' coincide; moreover, always under the same equivalence, both $j^{*}$ and $-\otimes K$ are left adjoint to $j_{*}$, hence they also coincide.

Thus all of this also holds for the corresponding derived categories; hence the statement follows since $D_{\mathrm{QCoh}}(X) =D(\mathrm{QCoh}(X))$ for $X$ quasi compact and separated.
\end{proof}

%%%%%%%%%%%%%%%%%%%%%%%%

\begin{corollary}\label{lift}
Let $X$ be a quasi compact, separated scheme over a field $k$.

Let $K=k(T)$ or $K=k(T,T')$.

The map
\begin{align*}
D_{\mathrm{QCoh}}(X_{K}) & \longrightarrow  (D_{\mathrm{QCoh}}(X))_{K} \\
C^{\bullet} &\mapsto (\mathrm{Forget}(C^{\bullet}),\rho_{C})
\end{align*}
is essentially surjective, where $\rho_{C}$ is the obvious $K$-structure on $C$.

Moreover, if $L$ is a finite separable extension of $K=k(T)$ with $L=K(\alpha)=K[T]/P(T)$ then we can lift an object $(C^{\bullet}, \rho_{C})\in (D_{\mathrm{QCoh}}(X))_{L}$ to an object $N^{\bullet}$ of $D_{\mathrm{QCoh}}(X_{K})$ endowed with a map $\tilde{\psi}\in \mathrm{End}(N^{\bullet})$ such that $P(\tilde{\psi})$ induces the zero map on all cohomology groups of $N^{\bullet}$.
\end{corollary}

\begin{proof}
By Lemma \ref{hartshorne}, there is an equivalence between $D_{\mathrm{QCoh}}(X_{K})$ and $D(\mathrm{QCoh}(X)_{K})$, hence it is sufficient to show that the map
\begin{align*}
D(\mathrm{QCoh}(X)_{K}) & \ra  (D(\mathrm{QCoh}(X)))_{K} \\
C^{\bullet} &\mapsto (\mathrm{Forget}(C^{\bullet}),\rho_{C})
\end{align*}
is essentially surjective.

Let $\mathcal{A}=\mathrm{QCoh}(X)$. This category satisfies AB5, hence theorem \ref{ess-sur} applies in this case.
\end{proof}

%%%%%%%%%%%%%%%%%%%%%%%%%%%%%%%%
%%%%%%%%%%%%%%%%%%%%%%%%%%%%%%%%
\section{A representability theorem for derived categories}%
%%%%%%%%%%%%%%%%%%%%%%%%%%%%%%%%
%%%%%%%%%%%%%%%%%%%%%%%%%%%%%%%%

The results of the previous section will become handy to study functors from $D^{b}_{\mathrm{Coh}}(X)$, where $X$ is defined over a field $k$, to a vector space over a bigger field in light of the following theorem:

\begin{theorem}\label{representable}
Let $k$ be a field, $\mathcal{A}$ be a $k$-linear abelian category satisfying AB5, and let $k\hookrightarrow K$ an inclusion of fields. Let $D(\mathscr{A})^{c}$ denote the full subcategory of compact objects in $D(\mathscr{A})$.

Given an exact, contravariant functor 
$$F:D(\mathscr{A})^{c}\ra \underline{\mathrm{mod}}_{K}$$
there exists a $T\in D(\mathscr{A})_{K}$ such that 
$$F(C)=\mathrm{Mor}_{D(\mathscr{A})_{K}}(C\otimes K,T)$$
for all $C\in D(\mathscr{A})^{c}$.
\end{theorem}

To prove this we will use the ideas from \cite[Lemma 2.14]{failure} where the version of this theorem with $k=K$ has been proved for a general triangulated category.

\begin{proof}[Proof of theorem \ref{representable}]
Let $D$ be the functor taking a $K$-vector space to its dual. Then $G=D\circ F$ is exact and covariant. Let $\mathscr{T}$ be the cocomplete triangulated subcategory generated by $D(\mathscr{A})^{c}$, i.e. the smallest full triangulated subcategory of $D(\mathscr{A})$ containing $D(\mathscr{A})^{c}$ which is closed under colimits.

Let $\tilde{G}:\mathscr{T}\ra\underline{\mathrm{Mod}}_{K}$ be the Kan extension of $G$ to $\mathscr{T}$: this is defined as
$$\tilde{G}(C)=\mathop{\mathrm{colim}}_{\substack{B \to C\\ B\in D(\mathscr{A})^{c}}} G(B)$$
Since $\tilde{G}$ is exact and commutes with coproducts, it follows that $D\circ\tilde{G}$ is exact and takes coproducts to products. Hence by the Brown representability theorem \cite[Theorem 8.3.3]{tricat} the functor $D\circ\tilde{G}$ is representable, as a functor to $\underline{\mathrm{Mod}}_{k}$, by an object $U\in \mathscr{T}\subset D(\mathscr{A})$.

%\cite[Theorem 3.1] {brown} only works in the case when A is Grothendieck!

The $K$-action on $\underline{\mathrm{Mod}}_{K}$ induces a $K$-action $\tilde{\rho}$ on $D\circ\tilde{G}=h_{U}$, hence by Yoneda we get a $K$-action $\rho$ on $U$, given by $K\stackrel{\rho}{\ra}\text{Nat}(h_{U},h_{U})=\text{Aut}(U)$. Therefore we obtain an object $(U,\rho)\in D(\mathscr{A})_{K}$. We need to show that 
$$D\circ\tilde{G}(C)=\text{Mor}_{D(\mathscr{A})_{K}}(C\otimes K,(U,\rho))$$
for all $C\in D(\mathscr{A})^{c}$.

To do so, first of all notice that as $k$-vector spaces
$$D\circ\tilde{G}(C)=\text{Mor}_{D(\mathscr{A})}(C,U)=\text{Mor}_{D(\mathscr{A})_{K}}(C\otimes K,(U,\rho))$$
because $K\otimes_{k}-$ is left adjoint to the functor forgetting the $K$-structure. By our definition of the $K$-action on $\text{Mor}_{D(\mathscr{A})}(C,U)$, this is the same as the $K$-action on $D\circ\tilde{G}(C)$; moreover the $k$-vector space map
\begin{align*}
\text{Mor}_{D(\mathscr{A})}(C,U) &\stackrel{\gamma}{\ra}\text{Mor}_{D(\mathscr{A})_{K}}(C\otimes K,(U,\rho))\\
f &\mapsto f\otimes \rho
\end{align*}
is compatible with the $K$-action since, for any $\alpha\in K$,
$$\gamma(\alpha\cdot f)=\gamma(\tilde{\rho}(\alpha)f)=\tilde{\rho}(\alpha)f\otimes\rho(\cdot)=f\otimes\rho(\alpha)\rho(\cdot)=\alpha\cdot(f\otimes\rho(\cdot))$$
hence we found that the two actions coincide and so 
$$D\circ\tilde{G}(C)=Mor_{D(\mathscr{A})_{K}}(C\otimes K,(U,\rho))$$

Let $T=(U,\rho)$. Now since $F$ is of finite type, we get
$$F(C)=(D\circ D\circ F)(C)=(D\circ G)(C)=(D\circ\tilde{G})(C)=Mor_{D(\mathscr{A})_{K}}(C\otimes K,T)$$
\end{proof}

%%%%%%%%%%%%%%%%%%%%%%%%%%%%%%%%%%%%
\begin{lemma}\label{beilinson}
Let $k$ and $K$ be two fields, $k\hookrightarrow K$.

Consider the equivalence of categories 
$$D^{b}(\underline{\mathrm{mod}}(\Lambda))\xrightarrow{\theta} D^{b}(\mathrm{Coh}(\pr^{n}_{k}))$$
as described in \cite{beilinson}.

Then there is also an equivalence of categories 
$$D^{b}(\underline{\mathrm{mod}}(\Lambda\otimes K))\xrightarrow{\theta_{K}} D^{b}(\mathrm{Coh}(\pr^{n}_{K}))$$
and the diagram
$$\xymatrix{
D^{b}(\underline{\mathrm{mod}}(\Lambda)) \ar[d] \ar[r]^-{\theta} & D^{b}(\mathrm{Coh}(\pr^{n}_{k}))\ar[d]\\
D^{b}(\underline{\mathrm{mod}}(\Lambda\otimes K)) \ar[r]^-{\theta_{K}} & D^{b}(\mathrm{Coh}(\pr^{n}_{K}))
}$$
is commutative.
\end{lemma}

\begin{proof}
By \cite{beilinson}, we have $\Lambda=\text{End}(\mathcal{M})$ where $\mathcal{M}=\bigoplus_{i=0}^{n}\Ol_{\pr^{n}_{k}}(i)$. Set $\mathcal{M}_{K}=\bigoplus_{i=0}^{n}\Ol_{\pr^{n}_{K}}(i)$, then
\begin{align*}
\text{End}_{\pr^{n}_{K}}(\mathcal{M}_{K}) &=\text{End}_{\pr^{n}_{K}}\left(\bigoplus_{i=0}^{n}\Ol_{\pr^{n}_{K}}(i)\right) =\bigoplus_{i,j=0}^{n}\text{End}_{\pr^{n}_{K}}\left(\Ol_{\pr^{n}_{K}}(i),\Ol_{\pr^{n}_{K}}(j)\right) = \\
&=\bigoplus_{i,j=0}^{n} K[x_{0},\ldots,x_{n}]_{j-i}=\bigoplus_{i,j=0}^{n} k[x_{0},\ldots,x_{n}]_{j-i}\otimes K \\
&=\left(\bigoplus_{i,j=0}^{n} k[x_{0},\ldots,x_{n}]_{j-i}\right)\otimes K=\Lambda \otimes K
\end{align*}

Moreover, the equivalence $\theta$ is induced by the map
$$\underline{\mathrm{mod}}(\Lambda)\xrightarrow{-\otimes_{\Lambda}\mathcal{M}} \text{Coh}(\pr^{n}_{k})$$
and if we let $h:\pr^{n}_{K}\ra \pr^{n}_{k}$ be the base change morphism, we obtain the following commutative diagram:
$$\xymatrix{
\underline{\mathrm{mod}}(\Lambda) \ar[rr]^{-\otimes_{\Lambda}\mathcal{M}} \ar[d]_{\otimes K}& & \text{Coh}(\pr^{n}_{k}) \ar[d]^{h^{*}} \\
\underline{\mathrm{mod}}(\Lambda\otimes K)\ar[rr]^{-\otimes_{\Lambda\otimes K}\mathcal{M_{K}} }& &\text{Coh}(\pr^{n}_{K})
}$$
this proves the last assertion.
\end{proof}

%%%%%%%%%%%%%%%%%%%%%%%%%%

We are now almost ready to prove Theorem \ref{ebounded}, but first we will prove the version of the theorem for the purely transcendental case. The following proof uses ideas from \cite[Theorem A.1]{generators}.

\begin{theorem}
\label{transcendental}
Let $X$ be a smooth projective variety over a field $k$. Let $K=k(T)$ or $K=k(T,T')$. Consider a contravariant, cohomological, finite type functor
$$H:D^{b}_{\mathrm{Coh}}(X)\ra \underline{\mathrm{mod}}_{K}$$
Then the complex $T$ of Theorem \ref{representable} lifts to a complex $S\in D^{b}_{\mathrm{Coh}}(X_{K})$ such that $H$ is representable by $S$, i.e. for every $C\in D^{b}_{\mathrm{Coh}}(X)$ we have
$$H(C)=Mor_{D^{b}_{\mathrm{Coh}}(X_{K})}(Lj^{*}C,S)$$
where $j:X_{K}\ra X$ is the base change morphism.
\end{theorem}

\begin{proof}
By Lemma \ref{representable}, the functor $H$ is representable by an element $T\in (D_{\mathrm{QCoh}}(X))_{K}$, i.e.
$$H(C)=Mor_{(D_{\mathrm{QCoh}}(X))_{K}}(C\otimes K,T)$$

Let $S$ be a lift of $T$ to $D_{\mathrm{QCoh}}(X_{K})$ (this is possible by Corollary \ref{lift}). Let $C$ be an element of $D^{b}_{\mathrm{Coh}}(X)$. By applying the functors in Lemmas \ref{ess-sur} and \ref{hartshorne} we get a $K$-linear map 

$$Mor_{D_{\mathrm{QCoh}}(X_{K})}(Lj^{*}C,S) \xrightarrow{\Psi(\cdot)} Mor_{(D_{\mathrm{QCoh}}(X))_{K}}(\Psi\circ Lj^{*}C,T)$$
and, since by Lemma \ref{hartshorne}, $\Psi\circ Lj^{*}C=C\otimes K$, we have
$$Mor_{(D_{\mathrm{QCoh}}(X))_{K}}(\Psi\circ Lj^{*}C,T)=Mor_{(D_{\mathrm{QCoh}}(X))_{K}}(C\otimes K,T)=H(C)$$

Hence to show that $H$ is represented by $S$ we just need to show that $\Psi(\cdot)$ is an isomorphism. It suffices to show that it is an isomorphism of $k$-vector spaces, which follows from the following diagram of $k$-vector spaces:
$$\xymatrix{
\mathrm{Mor}_{D_{\mathrm{QCoh}}(X_{K})}(Lj^{*}C,S)  \ar[r]^-{\Psi(\cdot)} \ar@{=}[d]
& \mathrm{Mor}_{(D_{\mathrm{QCoh}}(X))_{K}}(\Psi\circ Lj^{*}C,T) \ar@{=}[d] \\
\mathrm{Mor}_{D_{\mathrm{QCoh}}(X)}(C,Rj_{*}S)   \ar@{=}[d]
& \mathrm{Mor}_{(D_{\mathrm{QCoh}}(X))_{K}}(C\otimes K,T) \ar@{=}[d] \\
\mathrm{Mor}_{D_{\mathrm{QCoh}}(X)}(C,\text{Forget}(\Psi(S))) \ar@{=}[r] 
& \mathrm{Mor}_{(D_{\mathrm{QCoh}}(X))}(C,\text{Forget}(T)) 
}$$
here we used the fact that $Rj_{*}=\text{Forget}\circ \Psi$, again from Lemma \ref{hartshorne}.

So $\Psi(\cdot)$ is an isomorphism, and hence $H$ is represented by $S\in D_{\mathrm{QCoh}}(X_{K})$. We still have to show that $S$ is actually in $D^{b}_{\mathrm{Coh}}(X_{K})$.

Choose an embedding $\pi:X\ra\pr^{n}_{k}$. Let $H'=H\circ L\pi^{*}$. Let $\theta:D^{b}(\underline{\mathrm{mod}}(\Lambda))\ra D^{b}(\text{Coh}(\pr^{n}_{k}))$ and $\theta_{K}:D^{b}(\underline{\mathrm{mod}}(\Lambda\otimes K))\ra D^{b}(\text{Coh}(\pr^{n}_{K}))$ as defined in Lemma \ref{beilinson} above. Let $H''=H'\circ\theta$. Let $h:\pr^{n}_{K}\ra \pr^{n}_{k}$ be the base change morphism.

Consider the following diagram:
$$\xymatrix{
D^{b}(\underline{\mathrm{mod}}(\Lambda)) \ar[r]^-{\theta} \ar@/^4pc/ [rrr] ^{H''} \ar[d]_{-\otimes K}
&D^{b}(\text{Coh}(\pr^{n}_{k})) \ar[r]^-{L\pi^{*}} \ar@/^2pc/ [rr]^-{H'} \ar[d]_-{h^{*}} 
&D^{b}_{\mathrm{Coh}}(X) \ar[r]^-{H} \ar[d]_-{Lj^{*}} 
& \underline{Vect}_{K}\\
D^{b}(\underline{\mathrm{mod}}(\Lambda\otimes K)) \ar[r] ^-{\theta_{K}}  &D^{b}(\text{Coh}(\pr^{n}_{K})) \ar[r]^-{L\pi_{K}^{*}} 
&D^{b}_{\mathrm{Coh}}(X_{K}) 
}$$
and let $C\in D^{b}(\mathrm{Coh}(\pr^{n}_{k}))$.
\begin{align*}
H'(C) &=H(L\pi^{*}(C))=\text{Mor}_{D_{\mathrm{QCoh}}(X_{K})}(Lj^{*}L\pi^{*}C,S)= \\
&=\text{Mor}_{D_{\mathrm{QCoh}}(X_{K})}(L\pi_{K}^{*}h^{*}C,S)=\text{Mor}_{D_{\mathrm{QCoh}}(\pr^{n}_{K})}(h^{*}C,R\pi_{K*}S)
\end{align*}
so $H'$ is represented by $R\pi_{K*}S\in D_{\mathrm{QCoh}}(\pr^{n}_{K})$.

Let $V=\theta_{K}^{-1}(R\pi_{K*}(S))$ so that $H''$ is represented by $V$. Then
$$H''(\Lambda)=\text{Mor}_{\Lambda\otimes K}(\Lambda\otimes K,V)$$
and
\begin{align*} 
\sum_{n}\mbox{dim} H''(\Lambda[n]) &=\sum_{n} \mbox{dim Mor} (\Lambda[n]\otimes K,V)=\\
&=\sum_{n}\mbox{dim Mor} ((\Lambda\otimes K)[n],V)<\infty
\end{align*}
since $H''$ is of finite type. Therefore $V\in D^{b}(\underline{\mathrm{mod}}(\Lambda\otimes K))$.

This implies that $R\pi_{*}S\in D^{b}(\mathrm{Coh}(\pr^{n}_{K}))$ hence $S\in D^{b}(\mathrm{Coh}(X_{K}))$.
\end{proof}

%%%%%%%%%%%%%%%%%%%%%

\begin{proof}[Proof of theorem \ref{ebounded}]
The case where $L$ is purely transcendental of degree 2 over $k$ was treated in Theorem \ref{transcendental}. Let $L$ be a finitely generated separable field extension of $k$ with $\text{trdeg}_{k}L\leq 1$. There exists a field $K$ such that $K$ is a purely transcendental extension of $k$ of degree less than or equal 1, and $K\subset L$ is a finite extension. Set $L=K(\alpha)=K[T]/ P(T)$, where $P(T)$ is a separable polynomial. Consider the composition

$$\xymatrix{
D^{b}_{\mathrm{Coh}}(X) \ar[r]^{H} \ar@/^2pc/ [rr]^{H'} & \underline{\mathrm{mod}}_{L} \ar[r]^{\mathrm{Forget}} &\underline{\mathrm{mod}}_{K}
}$$

By theorem \ref{transcendental}, $H'$ is representable by an object $S\in D^{b}_{\mathrm{Coh}}(X_{K})$. Moreover, by Corollary \ref{lift}, $S$ is endowed with a map $\psi_{\alpha}$ such that $P(\psi_{\alpha})=0$ is zero on all the cohomology groups of $S$. 

First of all, this implies that there exists an $n$ such that $P(\psi_{\alpha})^{n}=0$. In fact, considering the good truncations $\tau_{\leq i}S$,
$$\ldots \to S^{i-1}\ra Z^{i}\to 0\text{ ,}$$ 
the claim follows inductively considering the distinguished triangles
$$\tau_{\leq i-1}S \to \tau_{\leq i}S \to H^{i}(S)[-i]$$
and recalling the fact that if $(f_{1},f_{2},f_{3})$ is a morphism between two triangles and two of the morphisms are nilpotent, then so is the third.

Now let $h:X_{L}\ra X_{K}$ be the base change morphism, and consider the pullback $Lh^{*}S\in D^{b}_{\mathrm{Coh}}(X_{L})$. It has an $L[T]$ action  induced by the morphism $Lh^{*}\psi_{\alpha}$, and $P(Lh^{*}\psi_{\alpha})^{n}=0$ so $Lh^{*}S$ has in fact an $L[T]/P^{n}(T)$-action. But since $P$ is a separable polynomial, the map $L[T]/P^{n}(T)\to L[T]/P(T)$ splits as $L$-algebra map hence $L[T]/P(T)$ also acts on $Lh^{*}S$.

Since, over $L$, $P(T)$ factors as $(T-\alpha)Q(T)$, we can find two elements $e_{1},e_{2}$ of $L[T]/P(T)$ such that  $e_{1}^{2}=e_{1}$, $e_{2}^{2}=e_{2}$, $e_{1}e_{2}=0$, $e_{1}+e_{2}=1$. But since $L[T]/P(T)$ acts on $Lh^{*}S$, this gives two idempotent operators $e_{1},e_{2}$ in $\text{Aut}_{D^{b}_{\mathrm{Coh}}(X_{L})}(Lh^{*}S)$ such that $e_{1}e_{2}=0$, $e_{1}+e_{2}=\text{id}_{Lh^{*}S}$.

Now since $D^{b}_{\mathrm{Coh}}(X_{L})$ is Karoubian by \cite[Proposition 3.2]{karoubian} we have obtained that $Lh^{*}S=E \oplus S_{2}$ and $Lh^{*}\psi_{\alpha}$ acts as multiplication by $\alpha$ on $E$.

We claim that $Rh_{*}E=S$. Consider the map
$$S \rightarrow Rh_{*}Lh^{*} S\xrightarrow{pr_{1}} Rh_{*}E$$
Under the identification $D_{\mathrm{QCoh}}(X_{L}) \stackrel{\Psi}{\longrightarrow} D(\text{QCoh}(X)_{L})$ this corresponds to $S\rightarrow S\otimes L\rightarrow \text{Forget} (E)$, so this is actually the identity map on $S$.

Then for every $C\in D^{b}_{\mathrm{Coh}}(X)$ we have a map of $L$-vector spaces
$$\text{Mor}_{D_{\mathrm{QCoh}}(X_{L})} (Lj^{*}C,E)\rightarrow \text{Mor}_{(D_{\mathrm{QCoh}}(X))_{L}}(C\otimes L,T)=H(C)$$
where $j:X_{L}\to X$ is the base change morphism, since $E$ is a lift of $S$ to $D^{b}_{\mathrm{Coh}}(X_{L})$ with the correct $L$-action. This map is an isomorphism because it is an isomorphism of $K$-vector spaces:
\begin{align*}
\text{Mor}_{D_{\mathrm{QCoh}}(X_{L})} (Lj^{*}C,E) &=\text{Mor}_{(D_{\mathrm{QCoh}}(X_{K}))}(Li^{*}C,Rh_{*}E) \\
&=\text{Mor}_{(D_{\mathrm{QCoh}}(X_{K}))}(Li^{*}C,S)=H'(C)
\end{align*}
where $i:X_{K}\to X$ is the base change morphism.
\end{proof}

%%%%%%%%%%%%%%%%%%%%%%%%%%%%%%%%%

\begin{proof}[Proof of theorem \ref{FM}]
Consider the composition
$$\xymatrix{
D^{b}_{\mathrm{Coh}}(X) \ar@/_2
pc/ [rrrr]_{H}  \ar[r]^-{F} & D^{b}_{\mathrm{Coh}}(Y) \ar[r]^-{i^{*}} & D^{b}_{\mathrm{Coh}}(\eta) \ar[r]^{H^{0}} & \underline{\mathrm{mod}}_{K(Y)} \ar[r]^{D} &\underline{\mathrm{mod}}_{K(Y)}
}$$

where $H^{0}(-)=H^{0}(\eta,-)$ and $D$ is the dual as $K(Y)$-vector space. $H$ is an exact contravariant finite type functor, hence by theorem \ref{ebounded} it is representable by $E\in D^{b}_{\text{Coh}}(X_{K(Y)})$.

Now consider the following diagram:
$$\xymatrix{
X_{K(Y)} \ar[r]^{p} \ar[d]^{g} \ar@/_2pc/[dd]_{j} &\eta \ar[d]^{i} \\
X\times Y \ar[r]^-{\pi_{2}} \ar[d]^{\pi_{1}} & Y \ar[d]\\
X\ar[r] &\text{Spec }k
}$$

Both squares are cartesian because $X_{K(Y)}=X\times_{\text{Spec }k} K(Y)=X\times Y \times_{Y}\eta$. Note that $g$ is a flat map, so the derived pullback is just regular pullback in every degree. Also, $g$ is an affine map so that pushforward is also exact. 

%$g$ is  flat because $i$ is: the inclusion of an integral domain in its quotient field is flat. $i$ is affine because it is given by the composition $\eta \to \eta\times Y \to Y$ and the first map is the base change of $Y\xrightarrow{\Delta}Y\times Y$, which is affine since $Y$ is separated, whereas the second map is the base change of $\eta \to Spec k$ and this is affine since $\eta$ is affine.

Let $E^{\vee}=\underline{\mathrm{RHom}}_{X_{K(Y)}}(E,\Ol_{X_{K(Y)}})$. Let us construct a complex $A\in D^{b}_{\mathrm{Coh}}(X\times Y)$ such that $Lg^{*}A=E^{\vee}\otimes \omega_{X_{K(Y)}}[\mbox{dim}X_{K(Y)}]$. 

Let $\mathcal{L}\in \text{Coh}(X\times Y)$ be a line bundle such that $E^{\vee}\otimes \omega_{X_{K(Y)}}[\mbox{dim}X_{K(Y)}]\otimes g^{*}\mathcal{L}^{\otimes n} \in D^{b}_{\text{Coh}}(X_{K(Y)})$ is generated by its global sections in each degree. Let $\{s_{i,\ell}\}$ be a set of generators in degree $\ell$. Consider the complex $g_{*}(E^{\vee}\otimes \omega_{X_{K(Y)}}[\mbox{dim}X_{K(Y)}]\otimes g^{*}\mathcal{L}^{\otimes n})$ on $D^{b}_{\mathrm{Coh}}(X\times Y)$. Then take the subcomplex  generated in each degree by $\{g_{*}s_{i,\ell}\} \cup \{g_{*}d s_{i,\ell-1}\} $, and twist it down by $\mathcal{L}^{-n}$. This gives the desired complex $A\in D^{b}_{\mathrm{Coh}}(X\times Y)$.

Then we get the following:
\begin{align*}
H^{0}\circ i^{*}\circ\Phi_{A}(C) &= H^{0} i^{*} R\pi_{2*}(A\otimes \pi_{1}^{*}C)\\
&=H^{0}Rp_{*} (g^{*}A\otimes g^{*} \pi_{1}^{*}C) \quad  \mbox{ (by flat base change)}\\
&= H^{0} Rp_{*} (E^{\vee}\otimes \omega_{X_{K(Y)}}[\mbox{dim }X_{K(Y)}]\otimes j^{*}C) \\ 
&= \mbox{Mor}(\Ol_{\eta},Rp_{*}(E^{\vee}\otimes \omega_{X_{K(Y)}}[\mbox{dim }X_{K(Y)}]\otimes j^{*}C)) \\
&= \mbox{Mor}(p^{*}\Ol_{\eta},E^{\vee}\otimes \omega_{X_{K(Y)}}[\mbox{dim }X_{K(Y)}]\otimes j^{*}C) \\
&=\mbox{Mor}(\Ol_{X_{K(Y)}},E^{\vee}\otimes \omega_{X_{K(Y)}}[\mbox{dim }X_{K(Y)}]\otimes j^{*}C) \\
&=\mbox{Mor}(E, \omega_{X_{K(Y)}}[\mbox{dim }X_{K(Y)}]\otimes j^{*}C) \\
&= D\circ \mbox{Mor} (j^{*}C,E) \\
&= D\circ H(C) \\
&= H^{0}\circ i^{*}\circ F(C)
\end{align*}
for every $C\in D^{b}_{\mathrm{Coh}}(X)$.

Now since $F$ is an exact functor,
\begin{align*}
H^{i}\circ i^{*}\circ F© &= H^{0} ( i^{*}\circ F(C) [i])=
H^{0} ( i^{*}\circ F(C[i]) )=\\
&=H^{0} ( i^{*}\circ\Phi_{A}(C[i]) )=
H^{0} ( i^{*}\circ\Phi_{A}(C) [i])=\\
&=H^{i}\circ i^{*}\circ\Phi_{A}(C) 
\end{align*}

Hence, since all cohomology groups agree and $D^{b}_{\mathrm{Coh}}(K(Y))$ is equivalent to the category of graded vector spaces over $K(Y)$, $F$ and $\Phi_{A}$ agree after restricting to the generic point of $Y$.
\end{proof}

\bibliography{refs}{}
\bibliographystyle{amsalpha} 

\end{document}